\documentclass[11pt]{amsproc}
\usepackage{latexsym,fancyhdr} 
\usepackage{amsmath,amssymb,xcolor}
\usepackage{multicol}
\usepackage{mathtools}
\usepackage{times}
\usepackage{pb-diagram} 
\usepackage[all,cmtip]{xy}
\usepackage{caption}
\usepackage{mathabx}
 \usepackage[toc,page]{appendix}
 
\usepackage{hyperref}

 \usepackage{amsmath, amssymb, latexsym, tikz}
\usepackage{caption,ifthen,url,fancyhdr,cite}
\usepackage[foot]{amsaddr}
\usepackage{libertine}
\usepackage{longtable}
\usepackage{upquote}
\usepackage{listings}
\usepackage{textcomp}  
\usepackage{adjustbox}

\usepackage{environ}
 \usepackage{fancyvrb,xcolor,fbox}
 \newcommand{\mygp}{{gap>}}

\newcommand{\mygpo}[1]{#1}

\usepackage{mathtools}
\usepackage{bm}
\usepackage{bbm}
\usepackage{xcolor}
\usepackage{multicol}

\usepackage{algorithm}
\usepackage[noend]{algpseudocode}

\newtheorem{Tma}{Theorem}[section]

\newtheorem{lemma}[Tma]{Lemma}

\newtheorem{proposition}[Tma]{Proposition}

\newtheorem{theorem}[Tma]{Theorem}
\theoremstyle{definition}
\newtheorem{definition}[Tma]{Definition}
\newtheorem{example}[Tma]{Example}

\newtheorem{remark}[Tma]{Remark}

\newcommand{\Z}{\mathbb{Z}}

\newcommand{\Aut}{{{\text{\rm{Aut}}}}}
\newcommand{\Comp}{{{\text{\rm{CP}}}}}

\newcommand{\can}{{\rm can}}
\newcommand{\tail}{{\rm tail}}

\newcommand{\doi}[1]{\href{https://doi.org/#1}{DOI: #1}}
\newcommand{\httpsurl}[1]{\href{https://#1}{\nolinkurl{#1}}}

\newcommand{\rev}[1]{#1} 

\newcommand{\newmu}{\rev{\sigma}}

\newcounter{ithmcount}
\newenvironment{iprf}{\begin{list}{{\rm
    \alph{ithmcount})}}{\usecounter{ithmcount}\labelwidth-5pt
      \leftmargin0pt \topsep3pt \itemsep1pt \parsep2pt}}{\end{list}}

\newenvironment{ithm}{\begin{list}{{\rm \alph{ithmcount})}}{\usecounter{ithmcount}\labelwidth18pt
      \leftmargin18
pt \topsep3pt \itemsep1pt \parsep2pt}}{\end{list}}

\newenvironment{items}{
\begin{list}{$\alph{item})$}
{\labelwidth18pt \leftmargin20pt \topsep3pt \itemsep5pt \parsep0pt}}
{\end{list}}

\title{Computing canonical labellings of finite solvable groups}

\author{Santiago Barrera Acevedo}
\address{La Trobe University, Department of Mathematical and Physical Sciences, Bundoora, VIC, Australia}
\email{S.BarreraAcevedo@latrobe.edu.au}

\author{Heiko Dietrich}
\address{Monash University, School of Mathematics, Clayton, VIC, Australia}
\email{heiko.dietrich@monash.edu}

\author{Max Horn}
\address{RPTU University Kaiserslautern-Landau, Germany}
\email{mhorn@rptu.de}

\thanks{The authors thank Bettina Eick, Derek Holt, and Eamonn A.\ O'Brien for comments. Dietrich was supported by a Universities Australia Grant in the Australia-Germany Joint Research Cooperation Scheme (\emph{The SmallGroups Library: New Classifications via Symbolic Computation}). Dietrich would like to thank the Isaac Newton Institute for Mathematical Sciences, Cambridge, for support and hospitality during the programme \emph{Latest trends in algebra and geometry}, where some of the work on this paper was undertaken; this work was supported by EPSRC grant EP/V521929/1.}

\renewcommand{\leq}{\leqslant}

\begin{document}

\vspace*{-1cm}

\begin{abstract}
We define a canonical labelling function on the class of finite solvable groups so that two such groups $G$ and $\tilde G$  are isomorphic if and only if $\can(G)=\can(\tilde G)$. Specifically, $\can(G)$ is a group presentation that describes a group isomorphic to $G$, and our description explains how to construct an isomorphism $G\to \can(G)$.  Our approach is motivated by O'Brien's \rev{(1994)} canonical presentations for finite $p$-groups and utilises ideas from group cohomology first described by Robinson (1982) and automorphism group algorithms developed by Smith (1994), Holt (2001), and others. We also discuss a proof-of-concept implementation for the computer algebra system GAP and comment on the major bottlenecks and open research questions.
\end{abstract}

\maketitle

\section{Introduction}\label{sec_intro}
A \emph{canonical labelling}  is a way of assigning a \emph{label}  to a mathematical object such that two objects receive the same label if and only if they are equivalent in some sense. For example, two complex square matrices are \emph{similar} if and only if they have the same Jordan Normal Form; two finite dimensional semisimple Lie algebras over the complex numbers are \emph{isomorphic} if and only if they have the same Dynkin diagram; two permutations are \emph{conjugate} in a symmetric group if and only if they have the same cycle type; and two connected, orientable, and compact surfaces are \emph{homeomorphic} if and only if they have the same genus.

Canonical labellings are important because they reduce the often difficult problem of checking equivalence to the simpler task of comparing labels for equality: Instead of asking whether two objects are equivalent, one computes canonical labels and asks whether these are equal. Canonical labellings of finite graphs are an established application that plays a crucial role in contemporary algorithms for solving the graph isomorphism problem, see \cite{nauty}. Building on these ideas, recent work of  Jefferson et al.\ \cite{orbit} describes a  framework for deciding the more general problem of whether two elements lie in the same orbit under the action of a permutation group. For a more general discussion of canonisation algorithms for combinatorial objects, we refer to Schweitzer and Wiebking \cite{schweitzer}. Applied to groups, Wiebking \cite{wiebking}  later demonstrated a canonical labelling that can be used to decide permutational isomorphisms.

 In this paper, we focus on canonical labellings for finite solvable groups and their practical applications in computational group theory, in particular, implementations for computer algebra systems. Since canonical labellings for group isomorphism would solve the group isomorphism problem, which is a famous and notoriously difficult problem (see \cite{DietrichWilson} and the references therein), one should not expect such a labelling to work effectively for all finite groups.
 Focus is therefore on special classes of groups; we give some examples. For the class of finite cyclic groups, the order of a group is a suitable canonical label. For the class of finite abelian groups, the \emph{abelian invariants} of a group can be used as  a canonical label; see \cite[Theorem 9.13]{HoltEick}. For the class of finite groups of size at most 2000 (except 1024), an example that is particularly relevant in practice is the ``IdSmallGroup'' label provided by the SmallGroups Library, see Besche et al.\ \cite{besche} and the implementations in GAP \cite{gap} and Magma \cite{magma}: If $G$ is a group of size $n$ at most $2000$ (not $1024$), IdSmallGroup$(G)=[n,i]$ if and only if $G$ is isomorphic to the $i$-th group in the SmallGroups Library list  of groups of size $n$.  This  illustrates the earlier observation that canonical labels are useful for reducing a list $G_1,\ldots,G_m$ of groups up to isomorphism: rather than performing up to $m(m-1)/2$ \emph{potentially expensive} isomorphism tests $G_i\cong G_j$, one computes the canonical labels $\ell_1,\ldots,\ell_m$ of the groups and performs up to $m(m-1)/2$ much cheaper equality checks $\ell_i=\ell_j$. This can be done even more efficiently if the labels can be sorted, for example, when they are integers.

 The examples of group labellings mentioned so far have in common that the labelling is determined by exhibiting sufficiently many isomorphism invariants to distinguish the groups. In particular, this approach does not produce explicit isomorphisms, which are often useful in practice. One noteworthy exception, and the main motivation for the present work, is provided by O'Brien's \cite{Obrien93} canonical presentations for finite $p$-groups resulting from the $p$-group generation algorithm; see also \cite[Section~9.4]{HoltEick}. O'Brien has described an algorithm that takes as input a group $G$ of prime power order and constructs a canonical group presentation $\can(G)$ and an isomorphism $G\to\can(G)$ such that two groups $G$ and $\tilde G$ of prime power order are isomorphic if and only if $\can(G)=\can(\tilde G)$, that is, if and only if the two presentations are \emph{equal} when considered as strings of symbols. For example, $\can(G)=\langle g_1,\ldots,g_r \mid g_1^p,\ldots,g_r^p, \; g_j^{g_i}=g_j\; (i<j)\rangle$ is one way to define a canonical presentation for an elementary abelian $p$-group $G$ of rank $r$. Since every finite nilpotent group is the direct product of groups of prime power order, the canonical presentations constructed in \cite{Obrien93} can easily be extended to canonical presentations for finite nilpotent groups.

\subsection{Main results}
The aim of this work is to describe and compute canonical labellings for finite solvable groups; denote this class of groups by $\mathcal{S}$. The usual way to describe a finite solvable group in computational group theory is via a consistent polycyclic group presentation (see Section \ref{sec_fsg}), and in practice we often assume that every group in $\mathcal{S}$ is given via such a presentation. This is not a serious restriction because effective methods exist to compute such a presentation for a finite solvable group $G$, assuming one can work effectively with $G$ to begin with.

Our main results are the following; let $G,\tilde G\in\mathcal{S}$ be finite solvable groups.
\begin{items}
\item[(A)] We describe a canonical presentation $\can(G)$ and prove that  $G\cong \tilde G$ if and only if $\can(G)=\can(\tilde G)$; see  Definition \ref{def_canpres}  and Theorem~\ref{thm_main}.
\item[(B)] We describe the construction of an isomorphism $G\to\can(G)$; see Section~\ref{sec_compiso}.
\item[(C)] We introduce our GAP package {\small\tt CanonicalPcPres} \cite{BADH}; see Section \ref{sec_comp}.
\end{items}

Our GAP package can be considered a proof-of-concept implementation of our algorithm, and while it works well for many examples, there are serious bottlenecks that are all related to difficult orbit calculations. We comment on these bottlenecks and opportunities for future research in Section \ref{sec_bottlenecks}.

We note that the class of finite solvable groups is an important class of groups. It contains  \emph{many} groups, for example, every group of odd order is solvable by the famous Odd-Order Theorem. It is also a class of groups that allows effective computations, see \cite[Chapter 8]{HoltEick} and  Section \ref{sec_fsg}. Lastly, similar to finite $p$-groups, every finite solvable group can be described by iterated group extensions by elementary abelian modules; this structure is the basis for our inductive approach, see Section~\ref{sec_chars}. The definition of canonical labellings for non-solvable groups will require different methods, cf.\ Hulpke's \cite{hulpke} work on perfect groups.

In Remark \ref{rem_pceasy} we explain that it is in theory easy to define \emph{some} canonical presentation for a finite solvable group, but the naive definition given there is difficult to compute and not practical. Our work is concerned with a more practical approach to constructing a canonical presentation together with an isomorphism.

\enlargethispage{1cm}

\subsection{General approach and structure of the paper}
The definition and construction of our canonical presentation is rather technical, which is why this section provides a top-level description of our approach; to keep this concise, we refer to Section \ref{sec_background} for all unexplained definitions.

Let $G$ be a finite solvable group. We assume that we can compute with $G$; this holds, for example, if $G$ is given by a permutation representation, polycyclic presentation, or reasonably small matrix representation. We first construct a certain characteristic subgroup series  $G=G_1>G_2>\ldots> G_{c+1}=\{1\}$ such that each section $G_i/G_{i+1}$ is elementary abelian. Our approach to construct $\can(G)$ is to iteratively construct epimorphisms $G\to \can(G/G_{i+1})$. The starting point for this iteration is the elementary abelian group $G_1/G_2$, for which a canonical presentation is defined in the introduction. Constructing $G_1/G_2$ and an epimorphism $G\to \can(G_1/G_2)$ can be achieved by the usual abelianisation algorithm via computing abelian invariants. For the next step, we describe $\can(G/G_{i+1})$ as a \emph{canonical} extension of $\can(G/G_i)$ by a module isomorphic to $G_i/G_{i+1}$. Similar ideas have been used by Smith \cite{SmithPhD}, Holt \cite{Holt01}, Cannon and Holt \cite{cannon}, and Howden \cite{howden} to describe algorithms for computing the automorphism group of a finite (solvable) group, see also \cite[Section~8.9]{HoltEick}. Some of these ideas date back to work of Robinson \cite{Robinson82} on group cohomology.

We provide some more details on this iteration. Assume an epimorphism from $G$ to $\can(G/G_i)$ has already  been constructed. This epimorphism induces an isomorphism $\pi\colon G/G_{i}\to \can(G/G_i)$. Since $G/G_{i+1}$ is an extension of $G/G_i$ by $G_i/G_{i+1}$, this allows us to describe $G/G_{i+1}$ as an extension of $\can(G/G_i)$ by $G_i/G_{i+1}$. Recall that $\can(G/G_i)$ is given by a canonical polycyclic presentation, and it turns out that the extension we seek can be described by a certain polycyclic presentation that \emph{extends} the one of $\can(G/G_i)$; we recall these details in Section~\ref{epcg}. To make this presentation canonical, we have to solve two problems: first, we need to choose a canonical action of $\can(G/G_i)$ on the module $G_i/G_{i+1}$; this is described in Section \ref{sec_module}. Once the module structure is fixed, we need to choose a canonical $2$-cocycle that describes the extension; this is discussed in Section \ref{sec_cps}. Throughout, ``canonical'' can be interpreted as ``deterministic'' such that the presentation $\can(G/G_i)$ depends only on the isomorphism type of $G$ for each $i$. Determining the canonical module structure and the canonical $2$-cocycle requires orbit calculations, which constitute the main practical bottleneck in our approach. We describe  these details in the main text.

The structure of the paper is as follows. In Section \ref{sec_background} we discuss the required technical background, including details on finite polycyclic groups and their polycyclic extensions. In Sections \ref{sec_solv} and \ref{sec_compiso}, we describe  $\can(G)$ and the construction of an isomorphism $G\to\can(G)$, respectively. Section \ref{sec_bottlenecks} discusses the main bottlenecks of our implementation and outlines ideas for future work. We conclude in Section \ref{sec_comp} with a brief discussion of our implementation.

\section{Technical background}\label{sec_background}
We recall the required information on finite solvable groups.  Most of this material  can be found in the Handbook of Computational Group Theory \cite{HoltEick}.

\subsection{Finite solvable groups}\label{sec_fsg}
A finite solvable group $G$ is polycyclic, that is, it admits a subnormal series \[G=G_1\geq G_2\geq \ldots \geq G_{n+1}=\{1\}\] such that each section $G_i/G_{i+1}$ is a finite cyclic group. A corresponding polycyclic generating set is a sequence $S=(g_1,\ldots,g_n)$ of elements \rev{$g_1,\ldots,g_n\in G$} such that each $G_i/G_{i+1}=\langle g_iG_{i+1}\rangle$; the  relative orders are $(r_1,\ldots,r_n)$ where each $r_i=|G_i/G_{i+1}|$. Every $g\in G$ can be uniquely written in the form $g=g_1^{e_1}\ldots g_n^{e_n}$ where each $e_i\in \{0,\ldots,r_i-1\}$; this is the normal form of $g$ with respect to $S$. By construction, $g_k^{r_k}=w_k(g_{k+1},\ldots,g_n)$ and $g_j^{g_i}=w_{i,j}(g_{i+1},\ldots,g_n)$ for each $i,j,k$ with $i<j$, where $w_k$ and $w_{i,j}$ are normalised words in the given arguments, \rev{that is, $w_k=g_{k+1}^{e_{k+1}}\ldots g_n^{e_n}$ is a normal form, and similarly for $w_{i,j}$}.
 
It follows that $G$ is isomorphic to the group given by the polycyclic presentation
\begin{eqnarray}\label{eq_pcp}\{g_1,\ldots,g_n \mid g_k^{r_k}=w_k\; (k=1,\ldots,n),\; g_j^{g_i}=w_{i,j}\; (1\leq i<j\leq n)\}.
\end{eqnarray}
By slight abuse of notation, we identify a group presentation with the group that it defines. We refer to \cite[Section 8]{HoltEick} for more details on polycyclic groups.
\enlargethispage{0.5cm}

\begin{remark}\label{rem_pceasy}
Every finite solvable group $G$ of order $p_1^{e_1}\ldots p_n^{e_n}$ (prime-power factorisation) has a polycyclic series whose sections are all of prime order, hence it can be represented by a polycyclic presentation with $n=e_1+\ldots+e_n$ generators. There are finitely many such presentations on generators $g_1,\ldots,g_n$,  and one could define $\can(G)$ to be that presentation that is \emph{minimal} among all such polycyclic presentations with respect to some ordering on the strings that represent these presentations. This approach obviously does not yield a practical algorithm.
\end{remark}

\subsection{Extensions of polycyclic groups}\label{epcg}
Let $G$ be a finite solvable group defined by the polycyclic presentation \eqref{eq_pcp}. Let $M$ be a $G$-module, given by the polycyclic presentation with polycyclic generators $(h_1,\ldots,h_m)$ and relative orders $(s_1,\ldots,s_m)$:
\begin{equation}\label{eq_pcpmod}
  \begin{aligned}M&=\langle \;h_1,\ldots,h_m \mid h_k^{s_k}=u_k\; (k=1,\ldots,m),\;\\ &\hspace*{2.87cm} h_j^{h_i}=h_j\; (1\leq i<j\leq m)\;\rangle,
  \end{aligned}
\end{equation}
where each $u_k=u_k(h_{k+1},\ldots,h_m)$ is a normal form. The $G$-action on $M$ is specified by normal forms $a_{i,j}=a_{i,j}(h_1,\ldots,h_m)$ in $M$ such that $h_j^{g_i}=a_{i,j}$ for all $g_i$ and $h_j$.

Let $E$ be an extension of $G$ by $M$, that is, a group that has a normal subgroup isomorphic to $M$ with quotient group $G$. We identify $M$ with a normal subgroup of $E$ and $G=E/M$ via an epimorphism $\pi\colon E\to G$ with kernel $M$. Consider a transversal $\tau\colon G\to E$, that is, a map such that $\pi(\tau(g))=g$ for all $g\in G$; we assume that $\tau$ is normalised, that is, $\tau(1)=1$. Every element in $E$ can be written as $\tau(g)a$ for unique $g\in G$ and $a\in M$. The transversal $\tau$ yields  a map \[\gamma\colon G\times G\to M,\quad (g,h)\mapsto \tau(gh)^{-1}\tau(g)\tau(h).\]
It is well-known that $\gamma$ is a $2$-cocycle in $Z^2(G,M)$ and that $E$ is isomorphic to the group $E(\gamma)$ defined as
\[E(\gamma)=\{(g,a) \mid g\in G, a\in M\}\]
with multiplication $(g,a)(h,b)=(gh,a^hb\gamma(g,h))$. Specifically, an isomorphism $E\to E(\gamma)$ can be defined  by $\tau(g)a\mapsto(g,a)$; see for example \cite[Proposition 5.8]{DietrichEick}. This isomorphism also shows that $E$ and $E(\gamma)$ are in fact \emph{equivalent extensions} of $G$ by $M$; see  \cite[p.\ 77]{DietrichEick}.

It is clear that $E(\gamma)$ is polycyclic with polycyclic generators \[((g_1,1),\ldots,(g_n,1),(1,h_1),\ldots,(1,h_m))\] and relative orders $(r_1,\ldots,r_n,s_1,\ldots,s_m)$. For every $g_k$ there is a normal form $t_k=t_k(h_1,\ldots,h_m)$ such that $(g_k,1)^{r_k}=(w_k,t_k)$; similarly, for $i<j$ there exists a normal form $t_{i,j}=t_{i,j}(h_1,\ldots,h_m)$ such that $(g_j,1)^{(g_i,1)}=(w_{i,j},t_{i,j})$. We collect all these words $t_k$ and $t_{i,j}$ in a sorted list, \rev{the \emph{tail vector} defined by $\gamma$: \[t=\tail(\gamma)=(t_1,\ldots,t_n,t_{1,2},t_{1,3},\ldots,t_{n-1,n});\]note that $t$ also depends on the chosen module structure (implicitly via $\gamma$).}

By slight abuse of notation,  $E(\gamma)$ is isomorphic to the group defined by the polycyclic presentation with polycyclic generating set $(g_1,\ldots,g_n,h_1,\ldots,h_m)$, relative orders $(r_1,\ldots,r_n,s_1,\ldots,s_m)$, and relations
\begin{equation}\label{relators}
\begin{array}{rllclll}
  g_k^{r_k}&=w_k t_k & (k=1,\ldots,n), \\[1ex]
  g_j^{g_i}&=w_{i,j} t_{i,j}& (1\leq i<j\leq n),\\[1ex]
  h_k^{s_k}& =u_k& (k=1,\ldots, m), \\[1ex]     h_j^{h_i}&=h_j& (1\leq i<j\leq m), \\[1ex]
  h_j^{g_i}& = a_{i,j}& (i=1,\ldots,n; j=1,\ldots,m).
\end{array}
\end{equation}
We denote this presentation by
\begin{eqnarray}\label{eq_etail}E(t)=E(\tail(\gamma)).\end{eqnarray}There is a natural isomorphism $E(\gamma)\to E(\tail(\gamma))$ that maps the generators of $E(\gamma)$ to  the generators of $E(\tail(\gamma))$; for details we refer to \cite[Section 8.7.2]{HoltEick}.

\rev{What follows is \cite[Lemma 8.47]{HoltEick}: the map that takes a $2$-cocycle $\gamma\in Z^2(G,M)$ to $\tail(\gamma)\in M^{(n^2+n)/2}$ is a group homomorphism;} we denote its image by $\mathcal{Z}(G,M)$. If $\gamma$ lies in the kernel of this homomorphism, then $E(\gamma)=G\ltimes M$ is a split extension, so $\gamma\in B^2(G,M)$ is a $2$-coboundary. It follows that the $2$-cohomology group $H^2(G,M)=Z^2(G,M)/B^2(G,M)$ is isomorphic to $\mathcal{Z}(G,M)/\mathcal{B}(G,M)$, where $\mathcal{B}(G,M)$ is the group generated by $\tail(\gamma)$ with $\gamma\in B^2(G,M)$.

It is important for our purposes that  $\mathcal{Z}(G,M)$ and $\mathcal{B}(G,M)$ can be computed without knowledge of $Z^2(G,M)$ and $B^2(G,M)$, see \cite[Section 8.7.2]{HoltEick}.

\enlargethispage{0.8cm}
\subsection{Characteristic series}\label{sec_chars}
Our iterative construction of $\can(G)$ requires us to use  a suitable subgroup series of $G$. Such a series is usually constructed by an iterative refinement, for example, starting with the derived series of $G$; see Example \ref{ex_solv}.  We denote the process of obtaining such a series a \emph{canonical filtration}.
\begin{definition}\label{def_css}A \emph{canonical filtration} is a process that takes as input a finite solvable group $G$ and outputs a subnormal subgroup series
  \[G=C_1(G)\geq C_2(G) \geq \ldots\geq C_{c+1}(G)=\{1\},\]
  where $c=c(G)$ depends on $G$, with the following properties:
  \begin{ithm}
   \item[(a)] the section $C_1(G)/C_2(G)$ is nilpotent, and each section $C_i(G)/C_{i+1}(G)$, for $i\geq 2$, is elementary abelian of prime-power order;
    \item[(b)] if $\tilde G$ is isomorphic to $G$, then $c(\tilde G)=c(G)$, and every isomorphism \mbox{$\varphi\colon G\to \tilde G$}  induces isomorphisms $C_i(G)\to C_i(\tilde G)$ for all $i$.
  \end{ithm}
 For a fixed filtration, the resulting subgroup series for $G$ is called the \emph{canonical characteristic subgroup series for $G$}.
\end{definition}
It follows from property (b) that each $C_i(G)$ is characteristic in $G$.

\begin{example}\label{ex_p}
  For a finite $p$-group $G$, the lower exponent-$p$ series is a canonical characteristic subgroup series for $G$. It is defined by  $C_1(G)=G$ and $C_{i+1}(G)=[G,C_i(G)]C_i(G)^{[p]}$ for $i\geq 1$, where $C_i(G)^{[p]}$ is the subgroup of $C_i(G)$ generated by all $p$-th powers and $[G,C_i(G)]$ is the subgroup generated by all commutators $[a,b]=a^{-1}b^{-1}ab$ with $a\in G$ and $b\in C_i(G)$; see \cite[p.\ 355]{HoltEick}. This series is the fastest descending series of $G$ with elementary abelian sections; it satisfies $C_i(G/C_j(G))=C_i(G)/C_j(G)$ for all $i\leq j$.
\end{example}

\begin{example}\label{ex_solv}
  For a finite solvable group $G$ that is not a $p$-group, we construct a canonical filtration as follows. We start with the derived series of $G$. We leave the first quotient as it is; this is the abelianisation $G/G'$, and the construction of the abelian invariants of $G$ allows us to define $\can(G/G')$ and construct an epimorphism $G\to \can(G/G')$. We refine every other section $S$  such that new subsequent sections are isomorphic to the Sylow subgroups of $S$ in ascending order of the primes. This yields a characteristic series \rev{in which every section (except possibly the first) is a group of prime power order}. Each of these \rev{(abelian)} $p$-group sections is now refined via their lower exponent-$p$ series, see Example \ref{ex_p}. This construction yields a canonical filtration. Modifications exist, for example, we can refine $G/G'$ further to arrange that every section in the series is elementary abelian. Alternatively, we can also impose that $C_2(G)$ is the nilpotent residual of $G$, that is, the unique smallest normal subgroup of $G$ with nilpotent quotient $G/C_2(G)$; this is useful in practice, see Section \ref{sec_comp}.   Another practical approach is to describe $G$ by a \emph{special polycyclic generating set} and use the associated subgroup series as described in \cite{specialpcgs}.
\end{example}

Throughout, we assume we have fixed a canonical filtration, for example, the one defined in Example \ref{ex_solv}, and refer to the resulting subgroup series as \emph{the} canonical characteristic subgroup series of~$G$. To simplify the notation, we usually denote this series by $G=G_1\geq G_2\geq \ldots\geq G_{c+1}=\{1\}$.

\section{Canonical presentations}\label{sec_solv}
The aim of this section is to define, for a finite solvable group $G$, a  polycyclic presentation $\can(G)$ such that $G\cong \tilde G$ if and only if $\can(G)=\can(\tilde G)$. Our construction also yields an explicit isomorphism  $\varphi\colon G\to \can(G)$, which we will discuss in Section \ref{sec_compiso}.

Let  $G=G_1\geq G_2\geq\ldots\geq G_{c+1}=\{1\}$ be the canonical characteristic series of~$G$, see Section \ref{sec_chars}. We  construct $\can(G)$ as an iterated extension by working up this characteristic  series. We assume that  $H=\can(G/G_{i})$ and an epimorphism $\pi\colon G\to H$ are constructed, and we then construct $\can(G/G_{i+1})$ as the \emph{canonical} extension of $G/G_i$ by $G_i/G_{i+1}$ that is isomorphic to $G/G_{i+1}$. For this we first discuss how to fix a \emph{canonical} module structure; the second step is then to fix a \emph{canonical} $2$-cocycle. We deal with these tasks in the subsequent~sections.

\subsection{Module structure}\label{sec_module}
Let $G$ be a finite solvable group with canonical characteristic series $G=G_1\geq\ldots \geq G_{c+1}=\{1\}$.  Let $\alpha\colon G\to \tilde G$ be an isomorphism and fix $k\in\{2,\ldots,c\}$. To simplify the notation, the induced isomorphisms $G/G_k\to \tilde G/\tilde G_k$ and $G_k/G_{k+1}\to \tilde G_k/\tilde G_{k+1}$ are also denoted by~$\alpha$.

Note that $G/G_{k+1}$ is an extension of $G/G_k$ by $G_k/G_{k+1}$, and $G/G_{k+1}$ acts on $G_k/G_{k+1}$ by conjugation yielding a \emph{module structure}  $G/G_{k+1}\to \Aut(G_k/G_{k+1})$. Since $G_k/G_{k+1}$ is abelian, this action induces a well-defined module structure  $G/G_k\to\Aut(G_k/G_{k+1})$. The next lemma compares this module structure with the one associated with the extension $\tilde G/\tilde G_{k+1}$.

\begin{lemma}\label{lem_mod1} Let $\alpha\colon G\to\tilde G$ be an isomorphism. If $\newmu\colon G/G_k\to \Aut(G_k/G_{k+1})$ and $\tilde \newmu\colon \tilde G/\tilde G_k\to \Aut(\tilde G_k/\tilde G_{k+1})$ are the module structures induced by conjugation, then for all $g\in G$\[\tilde \newmu(\alpha(g)\tilde G_k)=\alpha\circ\newmu(gG_k)\circ\alpha^{-1}.\]
\end{lemma}

\begin{proof}
  The automorphism $\alpha\circ\newmu(gG_k)\circ\alpha^{-1}\in \Aut(\tilde G_k/\tilde G_{k+1})$ maps $\alpha(h)\tilde G_{k+1}$ to $\alpha(h^g)\tilde G_{k+1}$, which is equal to $\alpha(h)^{\alpha(g)}\tilde G_{k+1}$; the claim follows.
\end{proof}

\begin{definition}\label{def_actAHM} Let  $A$ and $B$ be groups and $(u,v)\in \Aut(A)\times \Aut(B)$. Let $\kappa_v\colon \Aut(B)\to\Aut(B)$ be conjugation by $v^{-1}$. For a group homomorphism $\varepsilon\colon A\to \Aut(B)$ define $(u,v)\cdot\varepsilon = \kappa_v\circ\varepsilon\circ u^{-1}$, that is,
 \[(u,v)\cdot\varepsilon\;\colon\; A\to\Aut(B),\quad a\mapsto v\circ \varepsilon(u^{-1}(a))\circ v^{-1}.\]
This defines a group action of $\Aut(A)\times \Aut(B)$ on the set of all $A$-module structures on $B$.
\end{definition}

Now let $H=\can(G/G_k)=\can(\tilde G/\tilde G_k)$ and $M=C_p^d$ (the elementary abelian group of size $p^d$) such that $M\cong G_k/G_{k+1}$.  Choose isomorphisms
\begin{equation}\label{eq_isoms}
\begin{aligned}\varphi&\colon G/G_k\to H, &\hspace*{1cm}  \beta&\colon G_k/G_{k+1}\to M,\\
 \tilde \varphi&\colon \tilde G/\tilde G_k\to H, & \tilde \beta&\colon \tilde G_k/\tilde G_{k+1}\to M.
\end{aligned}
\end{equation}
These isomorphisms allow us to translate the module structures $\newmu$ and $\tilde \newmu$ described in Lemma \ref{lem_mod1} to $H$-module structures on $M$; we denote these by  $\varepsilon\colon H\to \Aut(M)$ and $\tilde\varepsilon\colon H\to \Aut(M)$: If $h\in H$, then \begin{eqnarray}\label{eq_ete}\varepsilon(h)=  \beta\circ\newmu(\varphi^{-1}(h))\circ\beta^{-1}\quad\text{and}\quad \tilde \varepsilon(h)=  \tilde \beta\circ\tilde\newmu(\tilde \varphi^{-1}(h))\circ\tilde\beta^{-1}.\end{eqnarray}
The next result shows that the $H$-module structures on $M$ induced by isomorphic groups lie in the same orbit under the action of $\Aut(H)\times \Aut(M)$ defined in Definition \ref{def_actAHM}.

\begin{proposition}\label{prop_mod}
  If $\tilde\varepsilon$ and $\varepsilon$ are as in \eqref{eq_ete}, then with the  notation of Definition \ref{def_actAHM}, the following hold.
  \begin{ithm}
  \item There is $(u,v)\in\Aut(H)\times\Aut(M)$ such that $\tilde\varepsilon=(u,v)\cdot \varepsilon$.
  \item If $(u,v)\in\Aut(H)\times\Aut(M)$, then $\hat\varepsilon=(u,v)\cdot \varepsilon$ is the $H$-module structure on $M$ defined by some group $\hat G$ isomorphic to $G$.
  \end{ithm}
\end{proposition}
\begin{proof}
 \begin{iprf}
\item Let $h\in H$ and recall from \eqref{eq_ete} that  $\newmu(\varphi^{-1}(h))=\beta^{-1}\circ \varepsilon(h)\circ \beta$.  Now Lemma \ref{lem_mod1} yields
\[\alpha\circ (\beta^{-1}\circ \varepsilon(h)\circ \beta)\circ\alpha^{-1}=\alpha\circ \newmu(\varphi^{-1}(h))\circ \alpha^{-1}=\tilde\newmu(\alpha(\varphi^{-1}(h))),\]
and therefore
\begin{align*}
  \tilde\varepsilon(h) &=  \tilde \beta\circ\tilde\newmu(\tilde \varphi^{-1}(h))\circ\tilde\beta^{-1}\\
  &=  \tilde \beta\circ\tilde\newmu(\alpha(\alpha^{-1}(\tilde \varphi^{-1}(h))))\circ\tilde\beta^{-1}\\
  &=\tilde \beta\circ (\alpha\circ \newmu(\alpha^{-1}(\tilde\varphi^{-1}(h)))\circ \alpha^{-1} )    \circ\tilde\beta^{-1}\\
  &= (\tilde \beta\circ \alpha\circ \beta^{-1})\circ\varepsilon [\varphi(\alpha^{-1}(\tilde\varphi^{-1}(h)))] \circ ( \beta   \circ \alpha^{-1}     \circ\tilde\beta^{-1}).
\end{align*}
Note that $v=\tilde \beta   \circ \alpha     \circ \beta^{-1}\in \Aut(M)$ and $u=\tilde\varphi \circ\alpha\circ \varphi^{-1}\in \Aut(H)$, and so the claim follows from
\[\tilde\varepsilon = \kappa_v \circ \varepsilon\circ u^{-1}=(u,v)\cdot\varepsilon.\]
\item As before, let $\varepsilon\colon H\to \Aut(M)$ be the module structure induced by $G$ with respect to isomorphisms $\varphi$ and $\beta$. Let $(u,v)\in\Aut(H)\times\Aut(M)$ and  define $\rev{\hat\varepsilon}=\kappa_v\circ\varepsilon \circ u^{-1}$. Let $\hat G=G$ and let $\hat\alpha\colon G\to \hat G$ be the identity. Now  define {$\hat\varphi=u\circ\varphi$ and $\hat\beta=v\circ\beta$.} The proof of part a) shows that $\hat \varepsilon\colon H\to\Aut(M)$ is defined by the group $\hat G$ and isomorphisms $\hat\alpha$, $\hat\varphi$, and $\hat\beta$.\qedhere
\end{iprf}
\end{proof}

The module structure $\varepsilon\colon H\to\Aut(M)$ is a group homomorphism defined on the canonical group $H=\langle x_1,\ldots,x_m\mid \mathcal{R}\rangle$, so determined by the matrix tuple $T=(\varepsilon(x_1),\ldots,\varepsilon(x_m))$. It is straightforward to define an ordering on the set of all module structures, for example, by representing such a matrix tuple as a single long vector and then using the lexicographical ordering. With respect to a fixed ordering, we define the following.

\begin{definition}\label{def_canmod}
The \emph{canonical $H$-module structure on $M$} is the  minimal element in the orbit of $\varepsilon$ under the action of  $\Aut(H)\times\Aut(M)$.
\end{definition}

\begin{remark}\label{remtup}
 The group $\Aut(M)$  acts by matrix conjugation on the matrix tuples  describing our module structures. Determining a canonical element in that orbit is the \emph{canonical forms for matrix tuples} problem, see \cite{cantup} for recent progress on the theoretical complexity of this problem.  Computing the orbit of a matrix tuple under $\Aut(H)\times\Aut(M)$ is one of the major bottlenecks of our practical implementation; we describe some attempts to improve practicality in Section \ref{sec_bottlenecks}.
\end{remark}

\begin{remark}\label{canmod}
  Let $H$, $M$, $\varphi$, $\beta$,  $\varepsilon$ be as in \eqref{eq_isoms} and \eqref{eq_ete}. Let $(u,v)\in\Aut(H)\times \Aut(M)$ such that $\hat\varepsilon=(u,v)\cdot\varepsilon$ is the canonical module structure. The proof of  Proposition~\ref{prop_mod} shows that if  the isomorphisms  $\varphi$ and $\beta$ are replaced by {$\hat\varphi=u\circ\varphi$ and $\hat\beta=v\circ\beta$}, respectively, then the module structure $\varepsilon$ will be replaced by $\hat\varepsilon$.  Thus,  going forward, we always assume that the isomorphisms $\varphi$ and $\beta$ are chosen such that the induced $H$-module structure on $M$ is the canonical one, denoted $\varepsilon\colon H\to\Aut(M)$.
\end{remark}

\subsection{Canonical extensions}\label{sec_cps}
We now describe the construction of the canonical presentation. As before, let $G$ be a finite solvable group with canonical characteristic subgroup series $G=G_1\geq \ldots\geq G_{c+1}=1$ as defined in Section \ref{sec_chars}. We assume we have computed the group $H=\can(G/G_k)$ with epimorphism $\pi\colon G\to H$, and we want to construct $\can(G/G_{k+1})$ as the canonical extension of $H$ by  $M\cong G_k/G_{k+1}$ isomorphic to $G/G_{k+1}$. For this we first consider $G_k/G_{k+1}$ as an $H$-module and arrange that this module structure is the canonical one, see Definition \ref{def_canmod}. Having fixed this module structure, we then define $\can(G/G_{k+1})$ as the extension defined by a canonical tail vector in $\mathcal{Z}(H,M)$; cf.\ Section \ref{epcg}. We  explain this in more detail.

The  canonical presentation for the abelian group $G/G_2=G/G'$ can be defined as follows.
\begin{definition}\label{def_k2}
  Let  $A\cong C_e$ be a cyclic group and factorise $e=p_1\ldots p_\ell$ with primes $p_1\geq p_2\geq \ldots\geq p_\ell$. Then
  \[\can(A)=\langle g_1,\ldots,g_\ell \mid g_i^{p_i}=g_{i+1}\; (i=1,\ldots,\ell-1),\; g_\ell^{p_\ell},\; g_j^{g_i}=g_j\; (i<j)\rangle.\]
  Let $A$ be  an abelian group with abelian invariants $[e_1,\ldots,e_m]$. Then $A$ is isomorphic to $C_{e_1}\times\ldots\times C_{e_m}$,  and $\can(A)$ is defined as the polycyclic presentation on generators $g_1,g_2,\ldots $ describing the direct product of $\can(C_{e_1})$, \ldots, $\can(C_{e_m})$.
\end{definition}

Note that the abelian invariants of $G$ can be computed by applying the well-known abelian quotient algorithm, see \cite[Section 9.2]{HoltEick}; this also produces an epimorphism $G\to \can(G/G')=\can(G/G_2)$. Our approach is to lift this iteratively to canonical presentations for  the quotients $G/G_k$ for $k=3,4,\dots$, until we have determined $\can(G/G_{c+1})=\can(G)$.

{\bf Notation.} {\it In the following, we assume that $k\in\{3,\ldots,c\}$ and  we have an epimorphism $\pi\colon G\to H$  with  induced isomorphism  $\varphi\colon G/G_k\to H$, defined by $gG_k\mapsto \pi(g)$.  By construction, $G_k/G_{k+1}$ is an elementary abelian  $G$-module, say of size $p^m$. We choose a basis of this module and define an isomorphism \[\beta\colon G_k/G_{k+1}\to M\quad\text{where}\quad M=\langle a_1,\ldots,a_m\mid  a_1^p,\ldots,a_m^p, a_j^{a_i}=a_j \;(i<j)\rangle.\]
Following Remark \ref{canmod}, we assume that $\varphi$ (hence $\pi$) and $\beta$ are chosen such that the induced module structure $\varepsilon\colon H\to \Aut(M)$ is the canonical one. We fix this module structure in the following and for $b\in M$ and $h\in H$ abbreviate $b^h=b^{\varepsilon(h)}$.}

Fix a normalised transversal $\tau\colon G/G_k\to G/G_{k+1}$  and define  \[\gamma'\colon G/G_k\times G/G_k\to G_k/G_{k+1},\quad (a,b)\mapsto \tau(ab)^{-1}\tau(a)\tau(b).\]
Then $G/G_{k+1}$ is isomorphic to the group
\[E(\gamma')=\{(a,b) \mid a\in G/G_k, b\in G_k/G_{k+1}\}\]
with $(a,b)(c,d)=(ac,b^cd\gamma'(a,c))$ via  $G/G_{k+1}\to E(\gamma')$, $\tau(a)n\mapsto(a,n)$.

We now translate these constructions to $H$ and $M$, and define
\begin{eqnarray}\label{eq_coc}\gamma=\beta\circ \gamma'\circ (\varphi^{-1},\varphi^{-1})\colon H\times H\to M.
\end{eqnarray}
By construction, the following holds.
\begin{lemma}\label{lem_iso1}
The map $\gamma$ is a $2$-cocycle in $Z^2(H,M)$ that defines a group
\[E(\gamma)=\{(g,a): g\in H,a\in M\}\] with multiplication $(g,a)(h,b)=(gh,a^hb\gamma(g,h))$. This group is isomorphic to $G/G_{k+1}$ via $(g,a)\mapsto \tau(\varphi^{-1}(g))\beta^{-1}(a)$.
\end{lemma}

We identify  $M$ with $1\times M\leq E(\gamma)$. Lemma \ref{lem_iso1} shows that  $E(\gamma)\cong G/G_{k+1}$ and there is an isomorphism that maps $M\leq E(\gamma)$ to $G_k/G_{k+1}\leq G/G_{k+1}$. If we repeat the same construction for a group $\tilde G$ that is isomorphic to $G$, then  we obtain a $2$-cocycle $\delta\in Z^2(H,M)$ with \[E(\delta)\cong \tilde G/\tilde G_{k+1}\cong G/G_{k+1}\cong E(\gamma),\] and there is an isomorphism $E(\delta)\to E(\gamma)$ that maps $M\leq E(\delta)$ to $M\leq E(\gamma)$.
The isomorphism problem of such extensions can be solved via the group of compatible pairs, which is defined as
\[\Comp(H,M)=\{(\mu,\nu)\in \Aut(H)\times \Aut(M)\mid \forall h\in H, b\in M: \nu(b^h)=\nu(b)^{\mu(h)}\},\]
and its action on $Z^2(H,M)$ where  $(\mu,\nu)\in \Comp(H,M)$ acts on  $\psi\in Z^2(H,M)$ via
\[(\mu,\nu)\cdot \psi =  \nu\circ\psi\circ(\mu^{-1},\mu^{-1}).\]We refer to \cite[Section 5.1.3]{DietrichEick} for more details on compatible pairs; see also \cite{HoltEick,Robinson82,SmithPhD}. In particular, it is known that there is an induced action on  $H^2(H,M)$.  For the next result, see \cite[Proposition 5.13]{DietrichEick}.

\begin{proposition}
Let $\gamma,\tilde\gamma\in Z^2(H,M)$. There is an isomorphism $E(\tilde \gamma)\to E(\gamma)$ that maps $M\leq E(\tilde\gamma)$ to $M\leq E(\gamma)$ if and only if the classes $\tilde\gamma+B^2(H,M)$ and $\gamma+B^2(H,M)$ lie in the same orbit under the action of $\Comp(H,M)$.
\end{proposition}

\begin{remark}\label{rem_cp}
  Note that $(\mu,\nu)\in\Comp(H,M)$ if and only if  $\varepsilon(h)=\nu^{-1}\circ \varepsilon(\mu(h))\circ\nu$ for every $h\in H$, which is equivalent to $(\mu,\nu)\in\Aut(H)\times\Aut(M)$ stabilising the module structure $\varepsilon\colon H\to\Aut(M)$; thus, $\Comp(H,M)={\rm Stab}_{\Aut(H)\times\Aut(M)}(\varepsilon)$.
\end{remark}

The next result shows that the cohomology class of $\gamma$ depends only on the isomorphism type of $G$; this is an important property for our construction of $\can(G)$.

\begin{proposition}\label{cor_ind1}
With the previous notation, the construction of the $2$-cocycle $\gamma$ depends on the chosen transversal $\tau$ and the choice of isomorphisms $\varphi$ and $\beta$. The following hold for the cohomology class $\gamma+B^2(H,M)$.
  \begin{ithm}
  \item The $\Comp(H,M)$-orbit of $\gamma+B^2(H,M)$ is independent of the choices of $\varphi$, $\beta$,  $\tau$.
  \item  The $\Comp(H,M)$-orbit of $\gamma+B^2(H,M)$ depends only on the isomorphism type of $G$.
  \end{ithm}
\end{proposition}
\begin{proof}
 \begin{iprf}
 \item   Choosing a different transversal for $\tau$ changes $\gamma'$ (and so $\gamma$) by a $2$-coboundary, see \cite[Lemma 5.6]{DietrichEick}, which does not affect the cohomology class $\gamma+B^2(H,M)$. If $\varphi'$ and $\beta'$ are different choices for $\varphi$ and $\beta$ (that do not change the $H$-module structure $\varepsilon$), then $\varphi'=\mu\circ \varphi$ and $\beta'=\nu\circ \beta$ for some $(\mu,\nu)\in \Aut(H)\times\Aut(M)$, and, in fact,  $(\mu,\nu)\in \Comp(H,M)$ by Remark \ref{canmod}. The construction in \eqref{eq_coc} with $\varphi'$ and $\beta'$ now yields a  $2$-cocycle $\gamma'$ that satisfies $\gamma'=\nu\circ\gamma\circ (\mu^{-1},\mu^{-1})$, hence $\gamma'$ and $\gamma$ lie in the same $\Comp(H,M)$-orbit, as claimed.
 \item  Let $\psi\colon G\to \tilde G$ be an isomorphism, with induced isomorphisms $\psi_i\colon G/G_i\to \tilde G/\tilde G_i$ and $\psi_i'\colon G_i/G_{i+1}\to \tilde G_i/\tilde G_{i+1}$. Let $\tilde \varphi\colon \tilde G/\tilde G_k\to H$ and $\tilde \beta\colon \tilde G_k/\tilde G_{k+1}\to M$ be isomorphisms such that the induced $H$-module structure on $M$ is the canonical one. Let $\tilde \tau\colon \tilde G/\tilde G_k\to \tilde G/\tilde G_{k+1}$ be a transversal.  These maps allow us to define a $2$-cocycle $\tilde\gamma\in Z^2(H,M)$ as in \eqref{eq_coc}.  We now show that for $G$ we can choose a specific transversal $\tau$ and isomorphisms $\varphi$ and $\beta$ such that the corresponding $2$-cocycle $\gamma$ equals $\tilde\gamma$. Recall from a)  that we can change $\tau$, $\varphi$, and $\beta$, without changing the $\Comp(H,M)$-orbit. Specifically, for $G$ we now consider the  isomorphisms $\varphi=\tilde\varphi\circ \psi_k$ and $\beta=\tilde\beta\circ \psi_k'$, and transversal $\tau= \psi_{k+1}^{-1}\circ\tilde\tau\circ\psi_k$. We show with a direct calculation that these maps define the $2$-cocycle $\gamma=\tilde\gamma$; let $a,b\in H$ and recall the definition of~$\gamma$:
  \begin{eqnarray*}
    \gamma(a,b) &=& \beta\big[\gamma'(\varphi^{-1}(a),\varphi^{-1}(b))\big]\\
    &=& \tilde\beta(\psi_k'\big[  \gamma' (\; \psi_k^{-1}(\tilde\varphi^{-1}(a)),\; \psi_k^{-1}(\tilde\varphi^{-1}(b))\;)\big]).
  \end{eqnarray*}
  By definition, $\gamma'(u,v)=\tau(uv)^{-1}\tau(u)\tau(v)$ and $\tau(x)=\psi_{k+1}^{-1}(\tilde\tau(\psi_k(x)))$; recall also that $\psi_k'\circ\psi_{k+1}^{-1}(u)=u$ for all $u\in \tilde G/\tilde G_{k+1}$, and so
   \begin{eqnarray*}
     \gamma(a,b) &=&  \tilde\beta(\psi_k'\big[
                       \psi_{k+1}^{-1}(\tilde\tau(\psi_k(\psi_k^{-1}(\tilde\varphi^{-1}(ab)))))^{-1}\\
      &&\hspace*{0.9cm}\psi_{k+1}^{-1}(\tilde\tau(\psi_k(\psi_k^{-1}(\tilde\varphi^{-1}(a)))))\\
      &&\hspace*{0.9cm}\psi_{k+1}^{-1}(\tilde\tau(\psi_k(\psi_k^{-1}(\tilde\varphi^{-1}(b)))))\big])\\
     &=& \tilde\beta( \; \tilde\tau(\tilde\varphi^{-1}(ab))^{-1}\tilde\tau(\tilde\varphi^{-1}(a))\tilde\tau(\tilde\varphi^{-1}(b))     \;)\\
     &=& \tilde\beta(\; \tilde\gamma'(\tilde\varphi^{-1}(a),\tilde\varphi^{-1}(b)))\\
     &=& \tilde\gamma(a,b).
   \end{eqnarray*}
   It remains to show that  $\varphi$ and $\beta$ define the same (canonical) $H$-module structure on $M$ as $\tilde\varphi$ and $\tilde\beta$.  If $\newmu$ denotes the conjugation action of $G/G_k$ on $G_k/G_{k+1}$, and $\tilde \newmu$ denotes the conjugation action of $\tilde G/\tilde G_k$ on $\tilde G_k/\tilde G_{k+1}$, then Lemma \ref{lem_mod1} shows that  $\tilde\newmu(\psi_k(g))=\psi_k'\circ \newmu(g)\circ \psi_k'^{-1}$ for all $g\in G/G_k$. Now  \eqref{eq_ete} shows that the module action defined by $\tilde\varphi$ and $\tilde\beta$ maps $h\in H$ to
   \begin{eqnarray*}
     && \tilde\beta\circ \tilde\newmu(\tilde\varphi^{-1}(h))\circ\tilde\beta^{-1}
     = \tilde\beta\circ   \psi_k'\circ \newmu(\psi_k^{-1}(\tilde\varphi^{-1}(h)))\circ \psi_k'^{-1}     \circ\tilde\beta^{-1}
   \end{eqnarray*}
which is the same as the module structure defined by $\varphi=\tilde\varphi\circ \psi_k$ and  $\beta=\tilde\beta\circ \psi_k'$, which maps $h\in H$ to
    \begin{eqnarray*}
    &&   \beta\circ \newmu(\varphi^{-1}(h))\circ\beta^{-1}= \tilde\beta\circ \psi_k'\circ \newmu(\psi_k^{-1}(\tilde\varphi^{-1}(h)))\circ \psi_k'^{-1}\circ\tilde\beta^{-1}.
    \end{eqnarray*}
    This completes the proof.\qedhere
 \end{iprf}
\end{proof}

It follows from the previous discussion that the $2$-cocycles in $Z^2(H,M)$ that define extensions isomorphic to $G/G_{k+1}$ with $M$ corresponding to $G_k/G_{k+1}$ are exactly the $2$-cocycles in the union of all cosets $\nu\circ\gamma\circ(\mu^{-1},\mu^{-1})+B^2(H,M)$ where $(\mu,\nu)$ runs over the elements in $\Comp(H,M)$. Moreover, this union only depends on the isomorphism type of $G$. It is a finite set of maps on finite sets, and therefore can be furnished with an ordering that allows us to pick a ``minimal element'' $\lambda$ in a canonical way, see also Section \ref{sec_compiso}. As canonical presentation for $G/G_{k+1}$ we now set
\[\can(G/G_{k+1}) =E(\tail(\lambda)),\]
 where $E(\tail(\lambda))$ is as defined in \eqref{eq_etail}. We use this to define $\can(G)$ as follows.

\begin{definition}\label{def_canpres}
  Let $G$ be a finite solvable group with canonical characteristic subgroup series $G=G_1\geq \ldots \geq G_{c+1}=\{1\}$. Starting with $\can(G/G_2)$ as in Definition \ref{def_k2}, we iterate the previous constructions of $\can(G/G_k)$ for $k=2,3,\ldots,c+1$, and set $\can(G)=\can(G/G_{c+1})$.
\end{definition}

In each iteration, we first determine the canonical module structure and then the canonical $2$-cocycle. The results discussed in this section imply the following.

\begin{theorem}\label{thm_main}
If $G$ and $\tilde G$ are finite solvable groups, then    $G\cong \tilde G$ if and only if $\can(G)=\can(\tilde G)$.
\end{theorem}

\begin{remark}
We compare our approach with the computation of a canonical presentation for finite $p$-groups described in \cite{Obrien93}. The approach in \cite{Obrien93} also uses an iteration, namely,  along the lower exponent-$p$ central series. Assuming a canonical presentation for $H=G/P_k(G)$, it computes a canonical presentation of $G/P_{k+1}(G)$ by considering the $p$-cover $H^*$  of $H$ and constructing a canonical quotient of $H^*$ that is isomorphic to $G/P_{k+1}(G)$. This canonical quotient depends on the choice of a certain canonical \emph{allowable} subgroup of the $p$-multiplier of $H$, which also requires an orbit calculation. This algorithm is part of the ANUPQ algorithms (see \cite{anupq} for the GAP implementations) that provide highly efficient implementations for $p$-group generation, $p$-quotient computation, and computations of canonical presentations. Our approach does not involve a covering group, but it computes a canonical presentation by considering canonical tail vectors describing the given group as suitable extensions. The \emph{cover} constructions in \cite{DH} or \cite{EH} could potentially also be used to define a canonical presentation for finite solvable groups; however, these methods pose some additional technical challenges, which is why we have opted for the approach described in this work.
\end{remark}

\section{Constructing isomorphisms}\label{sec_compiso}
We comment on a few of the crucial steps of the algorithm for computing $\can(G)$ and explain how to construct the isomorphism $G\to \can(G)$ along the way. We continue with the notation of the previous section, that is,  we consider $G/G_{k+1}\cong E(\gamma)$ as an extension of the canonical presentation $H=\can(G/G_k)$ by the module $M\cong G_{k}/G_{k+1}$. The groups for the next iteration (which is necessary if $G_{k+1}\ne \{1\}$) are  $\tilde H=\can(G/G_{k+1})$ and $\tilde M\cong G_{k+1}/G_{k+2}$. We assume that isomorphisms are chosen such that the module structures are the canonical ones, see Definition \ref{def_canmod} and Remark \ref{canmod}. Recall that  $\tilde H=E(\tail(\lambda))$ for some $\lambda \in Z^2(H,M)$.

\medskip

{\bf Tail vectors instead of cocycles.} GAP provides functionality to compute the  spaces $\mathcal{Z}(H,M)$ and $\mathcal{B}(H,M)$ of all tail vectors; see Section~\ref{epcg}.  Importantly, $H^2(H,M)\cong \mathcal{Z}(H,M)/\mathcal{B}(H,M)$, and there is a \emph{deterministic} preimage function for the natural projection $\mathcal{Z}(H,M)\to \mathcal{Z}(H,M)/\mathcal{B}(H,M)$. After computing the tail vector $t\in \mathcal{Z}(H,M)$ corresponding to $\gamma\in Z^2(H,M)$, one can use GAP's functionality to enumerate the $\Comp(H,M)$-orbit of the class $t+\mathcal{B}(H,M)$ in  $\mathcal{Z}(H,M)/\mathcal{B}(H,M)$ and identify a canonical vector in that orbit (for example, the lexicographical least vector). We then take its (deterministic) preimage $\lambda\in\mathcal{Z}(H,M)$ to define $\can(G/G_{k+1})=E(\tail(\lambda))$. In Section \ref{sec_lared} we report on some ideas that aim to  avoid the enumeration of the complete orbit, for example, in the case that $M$ is a trivial $H$-module.

\medskip 

{\bf Iteration: the next epimorphism.} For the next iteration, we need to construct an epimorphism $G\to \tilde H$. This can be done as follows. Continuing with the notation used above, Lemma \ref{lem_iso1} yields an isomorphism \[G/G_{k+1}\to E(\gamma),\quad  \tau(g)a\mapsto (\varphi(g),\beta(a)).\]
By construction, $\gamma$ and $\lambda$ define cohomology classes in the same $\Comp(H,M)$-orbit, that is, there exists $(\mu,\nu)\in\Comp(H,M)$ and $\delta_f\in B^2(H,M)$ such that \[\nu\circ\gamma\circ(\mu^{-1},\mu^{-1})=\lambda+\delta_f,\] where $\delta_f$ denotes the $2$-coboundary defined by the map $f\colon H\to M$, that is,  \rev{$\delta_f(g,h)=f(gh)^{-1}f(g)^h f(h)$}  for $g,h\in H$. A short calculation confirms that this induces an isomorphism \begin{eqnarray}\label{eq_cobf}E(\gamma)\to E(\lambda),\quad (g,a)\to (\mu(g),\nu(a)f(\mu(g))),
\end{eqnarray}see also \cite[Proposition 5.13]{DietrichEick}. Putting everything together, this yields an isomorphism
\[G/G_{k+1}\to E(\lambda), \quad \tau(g)a\mapsto (\mu(\varphi(g)),\nu(\beta(a))f(\mu(\varphi(g)))),\]
and hence an induced epimorphism $G\to E(\lambda)$. Composing with the natural isomorphism from $E(\lambda)$ to $\tilde H=E(\tail(\lambda))$, see the discussion after \eqref{eq_etail}, yields the required epimorphism $G\to \tilde H$. The construction of $\delta_f$ is explained in the next paragraph.

\medskip

{\bf Iteration: the $2$-coboundary.} The orbit-stabiliser calculation of $\Comp(H,M)$ on $\gamma+B^2(H,M)$ yields our chosen minimal class $\lambda+B^2(H,M)$, the stabiliser \begin{eqnarray}\label{eq_stab}S&=&{\rm Stab}_{\Comp(H,M)}(\gamma+B^2(H,M)),\end{eqnarray}and an element $(\mu,\nu)\in \Comp(H,M)$ such that $\delta_f=\nu\circ\gamma\circ(\mu^{-1},\mu^{-1})-\lambda$ is a $2$-coboundary. In \eqref{eq_cobf}, we need the values $f(h_1),\ldots,f(h_n)$ to construct the isomorphism $E(\gamma)\to E(\lambda)$. We do not know $f$, but we can compute values $\tilde f(h_1),\ldots,\tilde f(h_n)$ of some function $\tilde f\colon H\to M$ that satisfies $\delta_f=\delta_{\tilde f}$; this can be done  by solving a linear equation system (cf.\ \cite[p.\ 251]{HoltEick}), and these values can be used to construct the required isomorphism.

\medskip

{\bf Iteration: the next automorphism group.} We also need to compute the automorphism group of $\tilde H$, and our algorithm provides most of the ingredients that are necessary for that; this follows the approach of \cite{SmithPhD,Holt01}. For this we first recall how to construct generators for $\Aut(E(\lambda))$, and then use the isomorphism $E(\lambda)\to E(\tail(\lambda))=\tilde H$ to translate these to generators for $\Aut(\tilde H)$. Let $S$ be as in \eqref{eq_stab}. Generators for $\Aut(E(\lambda))$ come in two types. The first type can be constructed by taking automorphisms $(h,a)\to (\alpha(h),\beta(a)\iota(h))$ defined by compatible pairs $(\alpha,\beta)\in S$ stabilising $\lambda+B^2(H,M)$, where $\iota\colon H\to M$ is a map satisfying $\beta\circ\lambda\circ(\alpha^{-1},\alpha^{-1})-\lambda=\delta_\iota$. The second type can be constructed by taking automorphisms $(h,a)\to (h,a\kappa(h))$ for $\kappa\in Z^1(H,M)$, see \cite[Proposition 5.14]{DietrichEick}. Note that $(\mu,\nu)$ conjugates $S$ to ${\rm Stab}_{{{\Comp}}(H,M)}(\lambda+B^2(H,M))$ and that $Z^1(H,M)$ consists of all maps $\kappa\colon H\to M$ with $\kappa(1)=1$ and $\kappa(ab)=\kappa(a)^b\kappa(b)$ for all $a,b\in H$, see \cite[p.\ 74]{DietrichEick}; GAP provides  tools to construct $Z^1(H,M)$.

\section{Main bottlenecks and open research questions}\label{sec_bottlenecks}

The major bottlenecks of our implementation are the following tasks,  which are all related to orbit-stabiliser calculations and often depend on the rank of the module $M$, the size of $\Aut(H)$, and the dimension of the space $H^2(H,M)$.
\begin{items}
\item[(T1)] Computing the canonical module structure: this is an orbit calculation of  $\Aut(H)\times\Aut(M)$ on matrix tuples, see Remark \ref{remtup}.
\item[(T2)] Computing  $\Comp(H,M)$: this is a stabiliser calculation which can be carried out along the orbit calculation in (T1), see Remark \ref{rem_cp}.
\item[(T3)] Computing the canonical $2$-cocycle: this is an orbit calculation of the group $\Comp(H,M)$ acting on $H^2(H,M)$.
\end{items}

As described in Section \ref{sec_compiso}, in order to construct our isomorphism, in the orbit calculations during (T1) and (T3) we not only have to identify a minimal element in the orbit, but we also require a transversal element that maps the given element to the minimal one. Often we also seek a generating set for the stabiliser; this is relevant, for example, in (T2) or for the computation of the automorphism group.

Each of these tasks (T1), (T2), and (T3) is an interesting and challenging research problem in its own right; e.g., recall that (T1) encompasses the matrix tuple equivalence problem \cite{cantup}, see Remark \ref{remtup}. We comment on (T1) and (T2) in Section \ref{sec_canmods}, and on (T3) in Section \ref{sec_can2cs}. We highlight these problems here to indicate scope for future research; any progress on these problems will be incorporated in future versions of our implementation \cite{BADH}. Another opportunity for future improvements is to exploit the characteristic structure of groups; this is already used for computing automorphism groups of finite $p$-groups, but not yet for isomorphism construction; see also the recent work \cite{forum}.

We note that GAP already provides some advanced functionality for computing with large orbits, see for example the GAP package Orb \cite{orb}. Also relevant is an algorithm of  Schwingel \cite{schwingel} that allows one to construct a canonical $G$-orbit representative for a $p$-subgroup $G\leq {\rm GL}_d(p)$ acting on the natural module, including stabiliser and transversal element; see \cite[Section 5.2]{autpgroup} for details. In the introduction we mentioned the work \cite{orbit} that discusses canonical elements in orbits; it  assumes permutation groups as acting groups, and therefore is not directly applicable to our situation.

\subsection{Computing the canonical module structure and compatible pairs}\label{sec_canmods}
The module structure $\varepsilon\colon H\to\Aut(M)$ is a  homomorphism on the canonical group $H=\langle x_1,\ldots,x_m\mid \mathcal{R}\rangle$, so determined by the tuple $T=(\varepsilon(x_1),\ldots,\varepsilon(x_m))$. If all $\varepsilon(x_i)$ are scalar matrices, then we define $\can_M(T)=T$. Otherwise, let $j\geq 1$ be the first index such that $\varepsilon(x_j)$ is not scalar. We then define $\can_M(T)$ as  the tuple $(M_1,\ldots,M_m)$, where $M_i=\varepsilon(x_i)$ for $i=1,\ldots,j-1$, $M_j=B^{-1}\varepsilon(x_j)B$ is the Frobenius Normal Form of $\varepsilon(x_j)$, and $(M_{j+1},\ldots,M_m)$ is the unique minimal element in the $C_{\Aut(M)}(M_j)$-orbit of $(B^{-1}\varepsilon(x_{j+1})B,\ldots,B^{-1}\varepsilon(x_m)B)$; the latter can also be constructed by an iteration over the entries of that tuple, but we skip these details here. This computation also allows us to construct generators of the stabiliser, which yields a generating set for the compatible pairs, see (T2).

\subsection{Computing the canonical $2$-cocycle}\label{sec_can2cs}
Here we consider the action of $\Comp(H,M)$ on the cohomology group $H^2(H,M)$. GAP already provides functionality to compute this action via a matrix representation of $\Comp(H,M)$, and this computation can be made more efficient by employing various improvements. For example, one can compute and exploit the structure of a composition series of the module, or of a composition series of the group (and apply the aforementioned algorithm of Schwingel to a normal $p$-subgroup); we refer to  \cite{Schanze} for a further discussion of this problem. Here we discuss another simplification, namely, for the case that the $H$-module structure on $M$ is trivial.

\enlargethispage{0.5cm}

\subsubsection{Trivial module structure: a reduction using linear algebra}\label{sec_lared}
Suppose  $M$ is a trivial $H$-module, so ${\Comp}(H,M)=\Aut(H)\times \Aut(M)$. In (T3), instead of enumerating the union of the $\Comp(H,M)$-orbit of $t + \mathcal{B}(H,M)$, we exploit that the action of $\Aut(M) = {\rm GL}_d(p)$ is well-understood: we now describe how to solve (T3) by acting with $\Aut(H)$ on canonical $\Aut(M)$-orbit representatives.

Let $t$ be a tail vector with entries $t_1,\ldots,t_m\in \Z_p^d$. The action of $A\in \Aut(M)$ on $t$ is by component-wise multiplication $t_i\mapsto t_iA$. Thus, if we consider $t$ as a matrix with rows $t_1,\ldots,t_m$, the task is to define a canonical element in the set $t\cdot \Aut(M)$. For this we determine integers $i_1<\ldots<i_r$, each as small as possible, such that the set $\{t_{i_1},\ldots,t_{i_r}\}$ spans the row space of $t$, where $r$ is the rank of $t$. We then extend this set deterministically to a basis of $\mathbb{Z}_p^d$ (for example, by adding standard basis vectors $e_{j_1},\ldots,e_{j_{d-r}}$ with $j_1<\dots<j_{d-r}$, each as small as possible) and let $A$ be the inverse of the matrix that has this basis as row vectors. Now define $\can_M(t)$ as the tail vector represented by the matrix $tA$; we also call it the $M$-canonical tail vector of $t$. Note that $tA$ has the property that it has the $s$-th standard vector $e_s$ in row $i_s$ for $s=1,\ldots,r$. The next lemma shows that $\can_M(t)$ is indeed  a canonical element in the set $t\cdot \Aut(M)$.

\begin{lemma}
Two tail vectors $t$ and $t'$ lie in the same $\Aut(M)$-orbit if and only if $\can_M(t)=\can_M(t')$.
\end{lemma}
\begin{proof}
  By construction, if $\can_M(t)=\can_M(t')$, then $tA=t' A'$ for some $A,A'\in {\rm GL}_d(p)$, and therefore  $t=t' A' A^{-1}$ and $t'$ lie in the same $\Aut(M)$-orbit.

  For the converse, suppose $t' = t A$ for some $A\in \Aut(M)$.  Let $i_1<\ldots<i_r$ be as in the construction of $\can_M(t)$. Since $A$ is invertible, the vectors $t_{i_1}A,\ldots,t_{i_r}A$ span the row space of $tA$, hence of $t'$. By the same argument, if $i'_1<\ldots<i'_r$ are as in the construction of $\can_M(t')$, then  $t_{i'_1}'A^{-1},\ldots,t_{i'_r}'A^{-1}$ forms a basis of the row space of $t$. By the minimality of the labels $i_1,\ldots,i_r$ and $i'_1,\ldots,i'_r$, we must have $i'_s = i_s$ for all $s=1,\ldots,r$. Thus, $\can_M(t')$ and $\can_M(t)$ both have $e_{1},\ldots,e_{r}$ in rows $i_1,\ldots,i_r$. Note that $\can_M(t)$ and $\can_M(t')$ lie in the $\Aut(M)$-orbit of $t$, so  there is $B\in \Aut(M)$ with $\can_M(t')=\can_M(t)B$. Since $\can_M(t')$ and $\can_M(t)$ have the same rows $i_1,\ldots,i_r$, namely, the standard basis vectors $e_1,\ldots,e_r$, it follows that $B$ fixes this basis of the row space of $\can_M(t)$ element-wise. This implies $\can_M(t')=\can_M(t)B=\can_M(t)$, as claimed.
\end{proof}

The $(\Aut(H)\times\Aut(M))$-action on tail vectors can now be reduced to an action of $\Aut(H)$ on the set of $M$-canonical tail vectors:  given $t'=\can_M(t)$ and an automorphism $\alpha\in \Aut(H)$, we act with $(\alpha,1)\in{\Comp}(H,M)$ on $t'$ as usual to obtain a tail vector $t''$, and then return $\can_M(t'')$ as the image of $t'$ under $\alpha$.

As mentioned in Section \ref{sec_compiso}, in practice we act not on tail vectors but on a space that represents $H^2(H,M)$. We  compute the $M$-canonical vector of $t\in H^2(H,M)$ as $\pi(\can_M(\pi^{-1}(t)))$, where $\pi\colon \mathcal{Z}(H,M)\to H^2(H,M)$ is the natural epimorphism. This is possible because in GAP the preimage function of $\pi$ is deterministic.

\section{Implementation and examples} \label{sec_comp}
We have implemented our algorithm for the computer algebra system GAP \cite{gap} in the package {\small\tt CanonicalPcPres} \cite{BADH}. In practice, for our canonical filtration we always impose that the first subgroup $G_2$ is the nilpotent residual of $G$, that is, the unique smallest normal subgroup of $G$ with nilpotent quotient $G/G_2$. Since every finite nilpotent group is the direct product of its Sylow subgroups, this allows us to define and compute $\can(G/G_2)$ by using the canonical presentations for groups of prime power order, see \cite{Obrien93}. Our current implementation therefore  focuses on the general solvable case and provides  the framework for constructing  canonical extensions for not necessarily trivial module structure.

Before we present some examples, we need to set the expectations regarding runtimes. (All runtimes presented below have been obtained on a computer with an Apple M4 CPU and 32GB RAM, using  GAP 4.15.1.) Even the highly specialised  ANUPQ algorithms with more than 22000 lines of C code \cite[p.\ 3]{pq} currently require about 24 seconds to compute the standard presentation of the group $C_{11}\ltimes (C_{11}\ltimes (C_{11}\times C_{11^2}))$ with IdSmallGroup $[11^5,86]$; this  has been computed via the command {\footnotesize{\tt EpimorphismStandardPresentation}} provided by \cite{anupq}. (We note that more bespoke and efficient algorithms for dealing with groups of order dividing $p^5$ are currently developed in  \cite{Schanze}; this is another reason why in our example computations below we focus on non-nilpotent groups.) The difficulty of computing canonical presentations is not so much determined by the size of the group, but heavily influenced by the largest  module rank and, consequently, orbit sizes. Our algorithm faces the same challenges, and so there are already \emph{small} orders where our algorithm requires a long time, but there are also \emph{large} orders where we can compute the canonical presentation quickly; see Section \ref{sec_examples} for explicit runtimes.

The purpose of our implementation is to provide the necessary framework for  our algorithm; if more efficient solutions and implementations for the bottlenecks become available, then these will be incorporated and made public via the online repository \cite{BADH} of our package. In this regard, we emphasise the following.

\begin{remark}\label{rem_code}
Our canonical presentation is  \emph{canonical} only in the sense that our algorithm produces a group presentation depending solely on the isomorphism type of the input group; this is achieved by selecting canonical orbit representatives with respect to some deterministic algorithm.  If new orbit methods are implemented in future versions of our package \cite{BADH}, then this might change the choice of canonical representative, hence  the canonical presentation. For this reason, when comparing (or referring to) canonical presentations, these should be computed using the same release version of  \cite{BADH}.
\end{remark}

We conclude with a brief remark on the natural question whether our algorithm makes any advances for the complexity of group isomorphism testing. The short answer is `no'.

\begin{remark}
  Investigating the complexity of group isomorphism testing is a famous problem in group theory and theoretical computer science. While in practice one is interested in algorithms that run in time polynomial in the \emph{input size} (usually a generating set of size $O(\log n)$ for groups of size $n$), it is not even known whether the isomorphism problem can be solved in time polynomial in the \emph{group size}; we refer to \cite{DietrichWilson} for some recent results and references. The main bottleneck  is expected to be isomorphism testing of finite $p$-groups; indeed, not even the case of $p$-groups of class $2$ is known to be polynomial in the group size, see \cite{qiao} and the references therein. The algorithm for canonical presentations of $p$-groups in \cite{Obrien93} does not improve this complexity, so neither does our algorithm.
\end{remark}
\enlargethispage{0.8cm}

\newcommand{\myisom}{IsomorphismCanonicalPcGroup}
\newcommand{\mygrp}{CanonicalPcGroup}

\subsection{Examples}\label{sec_examples}
The command {\small\tt \myisom} provided by our GAP package takes as input a finite solvable group $G$ given by a polycyclic presentation (a GAP object of type {\small\tt pc\!\!\! group}) and returns an isomorphism $G\to \can(G)$ whose image is the canonical presentation ({\small\tt pc\!\!\! group}) of $G$; the function {\small\tt CanonicalPcGroup} only returns $\can(G)$.

In Figure \ref{code_small} we consider two `small' groups; the first has IdSmallGroup $[1800,430]$, and it takes a few milliseconds to compute the canonical presentation and isomorphism. To compare canonical presentations efficiently, we use the GAP function {\small\tt CodePcGroup} to encode each presentation by an integer (\emph{code}): the code together with the order of the group uniquely determines a polycyclic presentation of that group; we refer to \cite[Section~3.3]{codeBE} or \cite[Section 6.1]{j128} for more details. In Figure \ref{code_small}, we compute $1000$ `random' copies of the group and illustrate that the canonical presentations are indeed all the same. Here and below we construct these random copies with the GAP command {\small\tt RandomSpecialPcgsCoded}: it returns a group given by a \emph{random} special polycyclic generating set \cite{specialpcgs}, which provides particularly efficient group arithmetic.

The second group considered in Figure \ref{code_small} has IdSmallGroup $[1296,3583]$. It has a large-rank elementary abelian normal subgroup, and here it takes about 7 seconds to compute the canonical presentation and isomorphism. (While this is slower than hoped for, compare it with the performance of \cite{Obrien93} for the group with Id\-SmallGroup $[11^5,86]$; see above.) We present this example to illustrate that the bottlenecks are real. However, the examples below show that our implementation works very well in some other cases.

We used the SmallGroups Library to run over all solvable (but non-nilpotent) groups of size at most $2000$ whose size is not divisible by any $p^5$, $p$ a prime. There are $33450$ such groups, and computing the canonical presentation and isomorphism takes on average less than 40ms. The most challenging size is $2^4.3^4$, see the group in  Figure \ref{code_small}; on average each of the $3399$ solvable but non-nilpotent groups of size $2^4.3^4$ can be done in 82ms. We also considered all solvable cubefree groups of size between $2001$ and $50000$; these groups are provided by the SmallGroups Library. There are $386623$ such groups, and computing the canonical presentation and isomorphism takes on average less than 30ms. The longest runtime was $3$ seconds for the group with IdSmallGroup $[ 36300, 489 ]$; however, on average, each of the $968$ solvable groups of size $36300$ is done in 70ms.

In Figure \ref{code_cf1}, we run over the $3087$ solvable groups of cubefree size $44100$. The cubefree groups of size at most $50000$ are part of the SmallGroups Library. It took our code $219$ seconds to compute all canonical presentations and isomorphisms (on average, 71ms per group). GAP's {\small\tt IdSmallGroup} required $105$ seconds to identify these $3087$ groups, but our algorithm does more as it also provides explicit isomorphisms. GAP's {\small\tt IsomorphismGroups} required close to $6$ hours to perform these computations. The  isomorphism algorithm for cubefree groups provided by the package Cubefree (see \cite{cfisom}) required $113$ seconds to construct isomorphisms, but it does not provide a canonical labelling. In this example, our proof-of-concept implementation is already on par with {\small\tt IdSmallGroup} and the bespoke command {\small\tt IsomorphismCubefreeGroups}.

In Figure \ref{code_large}, we consider three larger groups $G$, $K$, and $L$, each defined by its code and size. The group $G$ has cubefree size $2375789332049525700$ and derived length $2$; our implementation required less than one second to compute the canonical presentation and isomorphism. The group $K$ has size  $2^7.3^5.5^4.7.11$ and derived length $4$; our implementation took  about $3$ seconds to compute the canonical presentation and isomorphism. The group $L$ has derived length $5$ and size $2^{10}.3^7.5^2$; here our implementation took about $6$ seconds.

\begin{figure}[htpb]
\begin{Verbatim}[fontfamily=qcr, fontsize=\fontsize{8}{9}\selectfont,commandchars=\\\{\}]
\mygp{\bf G} := SmallGroup(1800,430);;  # 1800 = 2^3.3^2.5^2; 749 groups
\mygp StructureDescription(G);
\mygpo{"C3 x (C3 : ((C5 x C5) : (C4 x C2)))"}
\mygp isoG := {\myisom}(G); time;
\mygpo{[ f1*f2,f3,f2*f4,f4,f5,f6^2*f7^3,f6^2*f7^2 ] -> [ f1,f2,f3,f4,f5,f6,f7 ] }
\mygpo{{56} # runtime in ms}
\mygp ### compute 1000 random copies and compare the canonical presentations
\mygp code := CodePcGroup(\mygrp(G)); # please note Remark 6.1 
\mygpo{374453170123379004775891797023}
\mygp grps := List([1..1000],
                   x -> PcGroupCode(RandomSpecialPcgsCoded(G),Size(G)));;
\mygp ForAll(grps, H -> code = CodePcGroup( \mygrp(H) ) );
\mygpo{true}
\mygp ### a small order where the bottlenecks are visible
\mygp {\bf K} := SmallGroup(1296,3583);; # 1296 = 2^4.3^4; 3609 groups
\mygp StructureDescription(K);
\mygpo{"C2 x (((C3 x C3) : C2) x S3 x S3)"}
\mygp {\mygrp}(K);; time;
\mygpo{{6661} # runtime in ms}
\end{Verbatim}
\caption{Example computation for two small groups}\label{code_small}
 \end{figure}

\begin{figure}[htpb]
\begin{Verbatim}[fontfamily=qcr, fontsize=\fontsize{8}{9}\selectfont,commandchars=\\\{\}]
\mygp all  := AllSmallGroups(44100,IsSolvableGroup);; Size(all);
\mygpo{3087 # there are 3087 solvable groups of size 44100 = 2^2.3^2.5^2.7^2}
\mygp ### create two identical lists of copies of these 3087 groups
\mygp {\bf grps} := List(all, x->PcGroupCode(RandomSpecialPcgsCoded(x),Size(x)));;
\mygp {\bf grps2}:= List(grps,x->PcGroupCode(CodePcGroup(x),Size(x)));;
\mygp List(grps, \mygrp);;time;
\mygpo{219390 # total runtime in ms using our method}
\mygp 219390/3087.;
\mygpo{{71.069} # average runtime per group in ms}

\mygp ### compare against existing functionality
\mygp ids := List(grps2, IdSmallGroup);;time;
\mygpo{{105170} # total runtime in ms using IdSmallGroup}

\mygp iso := List([1..Size(all)],
     i -> IsomorphismGroups(grps2[i],SmallGroup(ids[i])));; time;
\mygpo{{21162593} # total runtime in ms (5+ hours) using IsomorphismGroups}

\mygp isocf := List([1..Size(all)],
     i -> IsomorphismCubefreeGroups(grps2[i], SmallGroup(ids[i])));; time;
\mygpo{{113647} # total runtime in ms using IsomorphismCubefreeGroups}
\end{Verbatim}
\caption{Example computation for solvable groups of size $44100$}\label{code_cf1}
 \end{figure}

\begin{figure}[htpb]
\begin{Verbatim}[fontfamily=qcr, fontsize=\fontsize{8}{9}\selectfont,commandchars=\\\{\}]
\mygp{\bf G}:=PcGroupCode(3061760908318575786777401223340795845627611568558036018
754144322048690618653729499043969075561463552897604777321723599477407881360
502026288875744307018177675018562730812484030738208528249788355944888134127
92152406292181068000009150915011717294501050327990305,2375789332049525700);
\mygpo{<{\bf pc group of size 2375789332049525700} with 18 generators>}
\mygp ### 2375789332049525700 = 2^2.3^2.5^2.7^2.11^2.17^2.19^2.29.31.47.101
\mygp {### derived series section: C66671 x C19, C937750624650 x C2}
\mygp \mygrp(G);;time;
\mygpo{891 # runtime in ms}

\mygp{\bf K} := PcGroupCode(64041269144971972988463840467065174232558746859226343
976022739389867965931214626136977800164680906884623405344757194855557070302
975691696199119716934460115513212906153279934567711709282129828294532859104
127290442912769971452937809379342675923063057026989555059271001159000131100
333, 1496880000);
\mygpo{<{\bf pc group of size 1496880000} with 18 generators>} 
\mygp {### 1496880000 = 2^7.3^5.5^4.7.11}
\mygp {### derived series sections: C924 x C6, C270 x C2, C10^2, C5}
\mygp \mygrp(K);;time;
\mygpo{2888 # runtime in ms}

\mygp{\bf L} := PcGroupCode(351841997224853303820211227494417716166702805106186170
892240002799379603820385254089278106691847456555782306685606850709462562357
153692892027605512257112222201989307411079199922115884105318239958775875629
260453202577423723063788516901602113920381737658248445811387516224181677363
169722077856494109821308015683831098372915100590272704101327783998939074774
829886158201863370662651729102498460700090022221936040507232517936122980042
881887460321857390850440067582054212959850406560773218970981953554657802022
057028354674987919580546795699492221918565149052961411505052074461629544760
59807128552453150699885439653613171493495425130191736832, 55987200);
\mygpo{<{\bf pc group of size 55987200} with 19 generators>}
\mygp {### 55987200 = 2^10.3^7.5^2}
\mygp {### derived series sections: C2^2, C3, C10^2, C3^6, C2^6}
\mygp \mygrp(L);;time;
\mygpo{6250 # runtime in ms}
\end{Verbatim}
\caption{Example computations for three larger groups with derived lengths 2, 4, and 5, respectively}\label{code_large}
\end{figure}

\enlargethispage{1cm}


\end{document}